\documentclass{amsart}
\usepackage{amsfonts, amsbsy, amsmath, amssymb}
\usepackage{tikz}

\newtheorem{thm}{Theorem}[section]
\newtheorem{lem}[thm]{Lemma}

\newtheorem{conj}[thm]{Conjecture}

\newtheorem{thm-con}[thm]{Theorem-Conjecture}
\numberwithin{equation}{section}

\theoremstyle{definition}

\allowdisplaybreaks

\newcommand{\f}{\Bbb F}
\newcommand{\fin}{f_{\text{\rm inv}}}
\newcommand{\SF}{\text{\rm SF}_n}

\begin{document}

\title[a conjecture about the binary multiplicative inverse function]{on a conjecture about the sum-freedom of the binary multiplicative inverse function}

\author[Xiang-dong Hou]{Xiang-dong Hou}
\address{Department of Mathematics and Statistics,
University of South Florida, Tampa, FL 33620}
\email{xhou@usf.edu}

\author[S. Zhao]{Shujun Zhao}
\address{Department of Mathematics and Statistics,
University of South Florida, Tampa, FL 33620}
\email{shujunz@usf.edu}



\begin{abstract}
A recent conjecture by C. Carlet on the sum-freedom of the binary multiplicative inverse function can be stated as follows: For each pair of positive integers $(n,k)$ with $3\le k\le n-3$, there is a $k$-dimensional $\Bbb F_2$-subspace $E$ of $\Bbb F_{2^n}$ such that $\sum_{0\ne\in E}1/u=0$. We confirm this conjecture when $n$ is not a prime.
\end{abstract}

\maketitle

%%%%%%%%%%%%%%%%%%%%%%%%%%%%%%%%%%%%%%%%%%%
%  section 1
%%%%%%%%%%%%%%%%%%%%%%%%%%%%%%%%%%%%%%%%%%%
\section{Introduction}

While studying certain mathematical questions with potential applications in cryptography, C. Carlet introduced the notion of sum-free functions \cite{Carlet-DCC-2024,Carlet-JC-2025}. A function $f:\f_{2^n}\to \f_{2^n}$ is said to be {\em $k$th order sum-free} if $\sum_{x\in A}f(x)\ne 0$ for every $k$-dimensional affine subspace $A$ of $\f_{2^n}$. (Throughout this paper, unless specified otherwise, subspaces and affine subspaces of $\f_{2^n}$ are $\f_2$-subspaces and $\f_2$-affine subspaces.) The $2$nd order sum-free functions are precisely {\em almost perfect nonlinear} (APN) functions, which were introduced in \cite{Nyberg-LNCS-1992} and have since been extensively studied for their immunity to differential cryptanalysis in block ciphers; see \cite{Carlet-2021} and the references therein. The multiplicative inverse function $\fin:\f_{2^n}\to\f_{2^n}$, defined by $\fin(x)=1/x$ for $x\in\f_{2^n}^*$ and $\fin(0)=0$, is known to be APN for odd $n$ \cite{Nyberg-LNCS-1994}, and this function has been used in the $S$-boxes of the Advanced Encryption Standard (AES) \cite{Daemen-Rijmen-1999}. Carlet \cite{Carlet-DCC-2024} systematically studied the values of $k$ for which $\fin$ is $k$th order sum-free, and he left open the following question which is referred to as Carlet's Conjecture.

\begin{conj}\label{C1.1}
For $3\le k\le n-3$, $\fin$ is not $k$th order sum-free.
\end{conj}

We know that $\fin$ is $1$st order sum-free (trivially), and it is $2$nd order sum-free if and only if $n$ is odd \cite{Nyberg-LNCS-1994}. Moreover, for $1\le k\le n-1$, $\fin$ is $k$th order sum-free if and only if it is $(n-k)$th order sum-free \cite{Carlet-DCC-2024}. Therefore, a positive solution to Conjecture~\ref{C1.1} would provide a complete determination of all values of $1\le k\le n-1$ for which $\fin$ is $k$th order sum-free. 

It is known that if $A$ is an affine subspace of $\f_{2^n}$ not containing $0$, then $\sum_{x\in A}1/x\ne 0$ \cite[Theorem~1]{Carlet-JC-2025}. Therefore, Conjecture~\ref{C1.1} is equivalent to the claim that for each pair $(n,k)$ with $3\le k\le n-3$, there is a $k$-dimensional subspace $E$ of $\f_{2^n}$ such that
\begin{equation}\label{1.1}
\sum_{0\ne u\in E}\frac 1u=0.
\end{equation}
We call such a subspace $E$ a {\em zero-sum subspace} of $\f_{2^n}$. A subspace $E$ with $\sum_{0\ne u\in E}1/u\ne 0$ is called a {\em non zero-sum subspace}. There are mainly two approaches to zero-sum subspaces. The first one is algebraic and combinatorial: one uses various secondary constructions to produce zero-sum subspaces with desired dimensions. The other approach is algebro-geometric: one uses the Land-Weil bound on the number of zeros of absolutely irreducible polynomials over finite fields to prove the existence of zero-sum subspaces. Using these methods, Conjecture~\ref{C1.1} has been confirmed under various conditions; see Section 2. However, the conjecture itself remains open. The main result of the present paper is the following

\begin{thm}\label{T1.2}
Conjecture~\ref{C1.1} is true when $n$ is not a prime.
\end{thm}

In Section~2, we review some previous results that are essential for the present paper. The proof of Theorem~\ref{T1.2} is given in Section~3; the proof relies on several previous results and a critical refinement of one of them. Theorem~\ref{T1.2} compares to an early weaker version in \cite[Theorem~3.6]{EHRZ} which states that Conjecture~\ref{C1.1} is true if the minimal prime divisor $l$ of $n$ satisfies $(l-1)(l+2)\le (n+1)/2$.

%%%%%%%%%%%%%%%%%%%%%%%%%%%%%%%%%%%%%%%%%%%
%  section 2
%%%%%%%%%%%%%%%%%%%%%%%%%%%%%%%%%%%%%%%%%%%
\section{Previous Results}

Sum-free functions were systematically investigated in several recent papers \cite{Carlet-DCC-2024, Carlet-JC-2025, Carlet-Hou, EHRZ, Hou-Zhao}, and much progress has been made towards resolving Conjecture~\ref{C1.1}. Here, we only gather the existing results that will be used in the proof of our main theorem.

For $n\ge 2$, let
\begin{equation}\label{2.1} 
\SF=\{1\le k\le n-1:\fin\ \text{is $k$th order sum-free on $\f_{2^n}$}\}
\end{equation}
and 
\begin{align}\label{2.2}
\mathcal K_n\,&=\{1,\dots,n-1\}\setminus\SF\\
&=\{1\le k\le n-1:\text{$\f_{2^n}$ contains a $k$-dimensional zero-sum subspace}\}.\nonumber
\end{align}
Therefore, Conjecture~\ref{C1.1} simply claims that $\{3,\dots,n-3\}\subset\mathcal K_n$ for $n\ge 6$.

\begin{thm}\label{T2.1} {\rm (\cite[Theorem~3]{Carlet-DCC-2024})} 
$k\in\mathcal K_n$ if and only if $n-k\in\mathcal K_n$.
\end{thm}

\begin{thm}\label{T2.2} {\rm (\cite[Corollary~5.2]{Carlet-Hou}, \cite[Corollary~3.7, Theorem~3.13]{EHRZ})}
Conjecture~\ref{C1.1} is true if $n$ is divisible by $2$ or $3$ or $5$.
\end{thm}

\begin{thm}\label{T2.3} {\rm (\cite[Theorem~3.3]{EHRZ})}
Assume that $l\ge 2$, $l\mid n$, and that $\f_{2^n}$ contains an $r$-dimensional zero-sum subspace $F$. Let $s=\dim_{\f_{2^l}}\f_{2^l}F$, where $\f_{2^l}F$ is the $\f_{2^l}$-span of $F$. Then for each $0\le t\le n/l-s$, $\f_{2^n}$ contains a $(tl+r)$-dimensional zero-sum subspace $F_t$ with $\dim_{\f_{2^l}}\f_{2^l}F_t=t+s$. In particular,
\[
\Bigl\{tl+r:0\le t\le\frac nl-s\Bigr\}\subset\mathcal K_n.
\]
\end{thm}

For $k>0$, let
\begin{equation}\label{2.3}
\Delta(X_1,\dots,X_k)=\left|
\begin{matrix}
X_1&\cdots&X_k\cr
X_1^2&\cdots&X_k^2\cr
\vdots&&\vdots\cr
X_1^{2^{k-1}}&\cdots&X_k^{2^{k-1}}\end{matrix}
\right|\in\f_2[X_1,\dots,X_k],
\end{equation}
and for $0\le i\le k$, let
\begin{equation}\label{2.4}
\Delta_i(X_1,\dots,X_k)=\left|
\begin{matrix}
X_1&\cdots&X_k\cr
\vdots&&\vdots\cr
X_1^{2^{i-1}}&\cdots&X_k^{2^{i-1}}\cr
X_1^{2^{i+1}}&\cdots&X_k^{2^{i+1}}\cr
\vdots&&\vdots\cr
X_1^{2^k}&\cdots&X_k^{2^k}\end{matrix}
\right|\in\f_2[X_1,\dots,X_k].
\end{equation}
As we will see in the next section, $\Delta\mid\Delta_i$. Set 
\begin{equation}\label{2.5}
F_k(X_1,\dots,X_k)=\frac{\Delta_1(X_1,\dots,X_k)}{\Delta(X_1,\dots,X_k)}\in\f_2[X_1,\dots,X_k].
\end{equation}

\begin{thm}\label{T2.4} {\rm (\cite[Proposition~8]{Carlet-JC-2025}, \cite[\S4]{Carlet-Hou})}
A $k$-dimensional subspace $E$ of $\f_{2^n}$ is a zero-sum subspace if and only if 
\[
F_k(u_1,\dots,u_k)=0,
\]
where $u_1,\dots,u_k$ is any basis of $E$.
\end{thm}

%%%%%%%%%%%%%%%%%%%%%%%%%%%%%%%%%%%%%%%%%%%
%  section 3
%%%%%%%%%%%%%%%%%%%%%%%%%%%%%%%%%%%%%%%%%%%
\section{Proof of Theorem~1.2}

It is well known (see for example \cite[Lemma~3.51]{Lidl-Niederreiter-FF-1997}) that 
\begin{equation}\label{3.1}
\Delta(X_1,\dots,X_k)=\prod_{0\ne(a_1,\dots,a_k)\in\f_2^k}(a_1X_1+\cdots+a_kX_k).
\end{equation} 
For all $0\le i\le k$, we have $\Delta(X_1,\dots,X_k)\mid\Delta_i(X_1,\dots,X_k)$ since $(a_1X_1+\cdots+a_kX_k)\mid\Delta_i(X_1,\dots,X_k)$ for all $0\ne (a_1,\dots,a_k)\in\f_2^k$. Let $u_1,\dots, u_l\in\f_{2^n}$ be linearly independent over $\f_2$ and treat them as being fixed. For $k>0$ define
\begin{equation}\label{3.2}
F_{k,l}(X_1,\dots,X_k)=F_{k+l}(X_1,\dots,X_k,u_1,\dots,u_l)=\frac{\Delta_1(X_1,\dots,X_k,u_1,\dots,u_l)}{\Delta(X_1,\dots,X_k,u_1,\dots,u_l)}.
\end{equation}
(Note that $F_{k,l}$ depends not only on $l$, but also on $u_1,\dots,u_l$.) Here are a few straightforward facts about $\Delta$, $\Delta_1$ and $F_{k,l}$:
\begin{itemize}
\item 
$v_1,\dots,v_k\in\f_{2^n}$ are linearly independent over $\f_2$ if and only if $\Delta(v_1,\dots,v_k)$ $\ne 0$.

\smallskip
\item
$v_1,\dots,v_k\in\f_{2^n}$ form a basis of a non zero-sum subspace of $\f_{2^n}$ if and only if $\Delta_1(v_1,\dots,v_k)\ne 0$.

\smallskip
\item 
$\deg\Delta(X_1,\dots,X_k,u_1,\dots,u_l)=2^{k+l}-2^l$.

\smallskip
\item 
If $u_1,\dots,u_l$ ($l>0$) is a basis of a non zero-sum subspace of $\f_{2^n}$, then
\begin{align*}
\kern2.5em &\deg\Delta_1(X_1,\dots,X_k,u_1,\dots,u_l)=2^{k+l+1}-2^{l+1}\quad\text{(see Lemma~\ref{L1} below)},\cr
&\deg F_{k,l}(X_1,\dots,X_k)=2^{k+l}-2^l,\cr
&\deg_{X_j}\! F_{k,l}(X_1,\dots,X_k)=2^{k+l-1},\quad 1\le j\le k.
\end{align*}

\smallskip
\item 
For each $1\le j\le k$, $\Delta_1(X_1,\dots,X_k,u_1,\dots,u_l)$ is a $2$-polynomial in $X_j$ whose coefficient of $X_j$ is nonzero, i.e., 
\[
\Delta_1(X_1,\dots,X_k,u_1,\dots,u_l)=\sum_{i\ge 0}a_iX_j^{2^i}, 
\]
where $a_i\in\f_{2^n}[X_1,\dots,X_{j-1},X_{j+1},\dots,X_k]$ and $a_0\ne 0$. Consequently, $\Delta_1(X_1,\dots,X_k,u_1,\dots,u_l)$ is separable in $X_j$, and hence $\Delta_1(X_1,\dots,X_k,$ $u_1,\dots,u_l)$ is square-free in $\f_{2^n}[X_1,\dots,X_k]$.

\smallskip
\item 
It follows from \eqref{3.2} that 
\begin{equation}\label{3.3}
\gcd(F_{k,l}(X_1,\dots,X_k),\,\Delta(X_1,\dots,X_k,u_1,\dots,u_l))=1
\end{equation}
since $\Delta_1(X_1,\dots,X_k,u_1,\dots,u_l)$ is square-free.
\end{itemize}

\begin{lem}\label{L1}
Assume that $u_1,\dots,u_l$ ($l>0$) is a basis of a non zero-sum subspace of $\f_{2^n}$. Then 
\[
\deg\Delta_1(X_1,\dots,X_k,u_1,\dots,u_l)=2^{k+l+1}-2^{l+1}.
\]
\end{lem}

\begin{proof}
We have 
\begin{align*}
\Delta_1(X_1,\dots,X_k,u_1,\dots,u_l)=\,&\left|\begin{matrix}
X_1^{2^{l+1}}&\cdots&X_k^{2^{l+1}}\cr
\vdots&&\vdots\cr
X_1^{2^{k+l}}&\cdots&X_k^{2^{k+l}}
\end{matrix}\right|\Delta_1(u_,\dots,u_l)\cr
&+\text{terms of lower degree},
\end{align*}
where $\Delta_1(u_,\dots,u_l)\ne 0$.
\end{proof}

\begin{lem}\label{L2}
Let $u_1,\dots,u_l\in\f_{2^n}$ be linearly independent over $\f_2$, where $l\le n-2$. Then there exists $u_{l+1}\in\f_{2^n}$ such that $u_1,\dots,u_{l+1}$ is a basis of an $(l+1)$-dimensional non zero-sum subspace of $\f_{2^n}$. 
\end{lem}

\begin{proof}
It suffices to show that there exists $u_{l+1}\in\f_{2^n}$ such that $\Delta_1(u_1,\dots,u_{l+1})\ne 0$. Consider $\Delta_1(u_1,\dots,u_l,X)\in\f_{2^n}[X]$. We have 
\[
\Delta_1(u_1,\dots,u_l,X)=X\Delta(u_1^{2^2},\dots,u_l^{2^2})+\text{terms of higher degree},
\]
where $\Delta(u_1^{2^2},\dots,u_l^{2^2})=\Delta(u_1,\dots,u_l)^4\ne 0$. Hence $\Delta_1(u_1,\dots,u_l,X)\ne 0$. Since $\deg\Delta_1(u_1,\dots,u_l,X)\le 2^{l+1}<2^n$, there exists $u_{l+1}\in\f_{2^n}$ such that $\Delta_1(u_1,\dots,u_{l+1})\ne 0$.
\end{proof}

It was proved in \cite{Carlet-Hou} that $F_k(X_1,\dots,X_k)=F_{k,0}(X_1,\dots,X_k)$ is absolutely irreducible (i.e., irreducible over the algebraic closure $\overline\f_2$ of $\f_2$) when $k\ge 3$. We extend this result to the following form as the first step towards the proof of Theorem~\ref{T1.2}.

\begin{lem}\label{L3.1}
Assume that $k\ge 2$, $l>0$, and $u_1,\dots,u_l$ is a basis of non zero-sum subspace of $\f_{2^n}$. Then $F_{k,l}(X_1,\dots,X_k)$ is absolutely irreducible.
\end{lem}

\begin{proof}
By Equations~(4.3) -- (4.6) in \cite{Carlet-Hou}, we have
\begin{equation}\label{3.4}
F_{k,l}(X_1,\dots,X_k)=C_{k+l-1}X_1^{2^{k+l-1}}+C_{k+l-2}X_1^{2^{k+l-2}}+\cdots+C_0X_1+C_{-1},
\end{equation}
where
\begin{align}\label{3.5}
&C_{k+l-1}=F_{k-1,l}(X_2,\dots,X_k),\\ \label{3.6}
&C_i=F_{k-1,l}(X_2,\dots,X_k)\frac{\Delta_i(X_2,\dots,X_k,u_1,\dots,u_l)}{\Delta(X_2,\dots,X_k,u_1,\dots,u_l)},\quad 0\le i\le k+l-2,\\ \label{3.7}
&C_{-1}=\Delta(X_2,\dots,X_k,u_1,\dots,u_l)^2.
\end{align}
Note that $\deg C_{k+l-1}=\deg F_{k-1,l}(X_2,\dots,X_k)=2^{k-1+l}-2^{l}>0$ since $k\ge 2$. By \eqref{3.3}, \eqref{3.5} and \eqref{3.7},
\begin{equation}\label{3.8}
\gcd(C_{k+l-1},C_{-1})=1.
\end{equation}
Thus $F_{k,l}(X_1,\dots,X_k)$ as a polynomial in $X_1$ over $\overline\f_2[X_2,\dots,X_k]$ is primitive. Let $f\in\overline\f_1[X_2,\dots,X_k]$ be any irreducible factor of $C_{k+l-1}$. Then $f\mid C_i$ for all $0\le i\le k+l-2$ (by \eqref{3.6}), $f\nmid C_{-1}$ (by \eqref{3.8}), and $f^2\nmid C_{k+l-1}$ (since $C_{k+l-1}=F_{k-1,l}(X_2,\dots,X_k)$ is square-free). By Eisenstein's criterion, $F_{k,l}(X_1,\dots,X_k)$ is irreducible in $\overline\f_2[X_1,\dots,X_k]$.
\end{proof}

For $f(X_1,\dots,X_k)\in\f[X_1,\dots,X_k]$, where $\f$ is any field, define
\[
V_{\f^k}(f)=\{(x_1,\dots,x_k)\in\f^k: f(x_1,\dots,x_k)=0\}.
\]

\begin{thm}\label{T3.2}
Let $l>0$ and $u_1,\dots,u_l\in\f_{2^n}$ be a basis of a non zero-sum subspace of $\f_2$. If
\begin{equation}\label{3.9}
n\ge 4l+7,
\end{equation}
then $u_1,\dots,u_l$ can be extended to a basis of an $(l+2)$-dimensional zero-sum subspace of $\f_{2^n}$.
\end{thm}

\begin{proof}
Let $q=2^n$. By Theorem~\ref{T2.4}, it suffices to show that there exists $(x_1,x_2)\in\f_q^2$ such that $F_{2,l}(x_1,x_2)=F_{2+l}(x_1,x_2,u_1,\dots,u_l)=0$ but $\Delta(x_1,x_2,u_1,\dots,u_l)\ne 0$, that is,
\begin{equation}\label{3.10}
\bigl|V_{\f_q^2}(F_{2,l})\setminus V_{\f_q^2}(\Delta(X_1,X_2,u_1,\dots,u_l))\bigr|>0.
\end{equation}
By the Hasse-Weil bound (as stated in \cite[Corollary~2.5]{Aubry-Perret-1993}), we have
\begin{align*}
\bigl|V_{\f_q^2}(F_{2,l})\bigr|\,&\ge q+1-(2^{2+l}-2^l-1)(2^{2+l}-2^l-2)q^{1/2}-2^{1+l}\cr
&=q-(3\cdot 2^l-1)(3\cdot 2^l-2)q^{1/2}+1-2^{1+l}\cr
&>q-(9\cdot 2^{2l}-9\cdot 2^l+2)q^{1/2}-q^{1/2}\cr
&=q-9\cdot 2^{2l}q^{1/2}+(9\cdot 2^l-3)q^{1/2}\cr
&>q-9\cdot 2^{2l}q^{1/2}+9\cdot 2^{2l}.
\end{align*}
(In the first line of the above estimation, $2^{1+l}$ is an upper bound for the number of points at infinity on the curve $F_{2,l}(X_1,X_2)=0$.) On the other hand, by Bezout's theorem \cite[\S5.3]{Fulton-1989}, 
\[
\bigl|V_{\f_q^2}(F_{2,l})\cap V_{\f_q^2}(\Delta(X_1,X_2,u_1,\dots,u_l))\bigr|\le(2^{2+l}-2^l)^2=9\cdot 2^{2l}.
\]
Therefore,
\[
\bigl|V_{\f_q^2}(F_{2,l})\setminus V_{\f_q^2}(\Delta(X_1,X_2,u_1,\dots,u_l))\bigr|>q-9\cdot 2^{2l}q^{1/2}=q^{1/2}(q^{1/2}-9\cdot 2^{2l}).
\]
From \eqref{3.9}, we have
\[
n\ge 4l+7>4l+2\log_29 \;(\approx 4l+6.34).
\]
Thus 
\[
q^{1/2}=2^{n/2}>9\cdot 2^{2l},
\]
which proves \eqref{3.10}.
\end{proof}

\begin{proof}[Proof of Theorem~\ref{T1.2}]
Let $l$ be the minimal prime divisor of $n$. By Theorem~\ref{T2.2}, we may assume $l\ge 7$, whence $n\ge 49$. In this case,
\[
n-4l-6\ge n-4\sqrt n-6>0,
\]
whence $n\ge 4l+7$. 

For each $3\le r\le l+2$, write $r=2+l'$, where $1\le l'\le l$. Let $u_1,\dots,u_{l'-1}\in\f_{2^l}$ be linearly independent over $\f_2$, and exetnd $u_1,\dots,u_{l'-1}$ to a basis $u_1,\dots,u_{l'}$ of an $l'$-dimensional non zero-sum subspace of $\f_{2^n}$ (by Lemma~\ref{L2}). By Theorem~\ref{T3.2}, $u_1,\dots,u_{l'}$ can be extended to a basis of an $r$-dimensional zero-sum subspace $F$ of $\f_{2^n}$. Note that $s:=\dim_{\f_{2^l}}\f_{2^l}F\le 4$. Thus by Theorem~\ref{T2.3},
\[
\Bigl\{tl+r:0\le t\le \frac nl-4 \Bigr\}\subset\mathcal K_n.
\]
Since 
\[
\Bigl(\frac nl-4\Bigr)l+r=n-4l+r\ge n-4\sqrt n+3\ge\frac 12(n-1),
\]
(here we used the fact that $n\ge 49$), we have
\[
\mathcal K_n\supset\Bigl\{tl+r: 3\le r\le l+2,\ 0\le t\le \frac nl-4 \Bigr\}\supset\Bigl\{3,4,\dots,\frac{n-1}2\Bigr\}.
\]
By Theorem~\ref{T2.1}, we conclude that
\[
\mathcal K_n\supset\{3,4,\dots,n-3\}.
\]
\end{proof}

%%%%%%%%%%%%%%%%%%%%%%%%%%%%%%%%%%%%%%%%%%%

%%%%%%%%%%%%%%%%%%%%%%%%%%%%%%%%%%%%%%%%%%%

\end{document}